\documentclass[12pt]{amsart}

\usepackage{amsmath, amssymb}

\newtheorem{theorem}{Theorem}

\newtheorem{cor}[theorem]{Corollary}

\theoremstyle{definition}

\theoremstyle{remark}
\newtheorem{remark}[theorem]{Remark}

\numberwithin{equation}{section}

\newcommand{\rr}{{\mathbb R}}
\newcommand{\nat}{{\mathbb N}}

\newcommand{\rd}{{\mathbb R^d}}
\newcommand{\Exp}{\mathbb E}
\newcommand{\SD}{\mathbb D}
\newcommand{\Ima}{\operatorname{Im}}
\newcommand{\GL}{\operatorname{GL}}
\newcommand{\Cov}{\operatorname{Cov}}

\allowdisplaybreaks

\begin{document}
\sloppy
\title[Multiparameter Transfer Theorem]{A General Multiparameter Version\\ of Gnedenko's Transfer Theorem}
\author{Peter Kern}
\address{Peter Kern, Mathematisches Institut, Heinrich-Heine-Universit\"at D\"usseldorf, D-40225 D\"usseldorf, Germany}
\email{kern\@@{}math.uni-duesseldorf.de}

\date{\today}

\begin{abstract}
Limit theorems for a random number of independent random variables are frequently called transfer theorems. Investigations into this direction for sums of random variables with independent random sample size have been originated by Gnedenko. We present a widely applicable transfer theorem for random variables on a general metric space with random multiparameters instead of random sample sizes. This summarizes an intrinsic principle behind the transfer type results known from the literature.
\end{abstract}


\maketitle

\baselineskip=18pt

\section{Introduction}

In 1969 Gnedenko and Fahim \cite{GF} proved the following limit theorem for sums of a random number of real-valued random variables; see also \cite{Gne} for the precise assumptions. For every $n\in\nat$ let $(X_{n,k})_{k\in\nat}$ be a sequence of independent and identically distributed (i.i.d.) random variables and let $(T_n)_{n\in\nat}$ be a sequence of positive integer valued random variables such that $T_n$ and $(X_{n,k})_{k\in\nat}$ are independent for every $n\in\nat$. If there exists a sequence $k_n\to\infty$ such that as $n\to\infty$
\begin{equation}\label{setting}
\sum_{k=1}^{k_n}X_{n,k}\Rightarrow\mu\quad\text{ and }\quad\frac{T_n}{k_n}\Rightarrow\rho
\end{equation}
for some probability distributions $\mu$ on $\rr$, respectively $\rho$ on $(0,\infty)$, then the random sum converges to a mixture distribution
\begin{equation}\label{gne}
\sum_{k=1}^{T_n}X_{n,k}\Rightarrow\int_0^\infty\mu^t\,d\rho(t),
\end{equation}
where ``$\Rightarrow$'' denotes convergence in distribution, i.e.\ the distribution of the random variable on the left-hand side converges weakly to the measure on the right-hand side, and $\mu^t$ denotes the $t$-fold convolution power of the necessarily infinitely divisible distribution $\mu$, well defined by the L\'evy-Khintchine formula. The result is frequently called the {\it transfer theorem of Gnedenko}, since randomization of the  sample size transfers the limits in \eqref{setting} to a randomized limit in \eqref{gne}. There is an ongoing considerable interest in random limit theorems of this kind from numerous applications and today there exists a vast literature on transfer theorems generalizing the above. Without demanding to give a complete list, we refer to \cite{Rob,Tho,KKru,FKT,KK1,KK2,DY} and the literature cited therein, to mention just a few direct generalizations. Comprehensive studies of random limit theorems of various kind are given in the monographs \cite{Gut,GK,BK,Sil}. Although there is an intrinsic principle behind all the transfer type results, a new proof is given in every new situation. Our aim is to capture this principle and to prove a version of the transfer theorem in a generality as broad as possible. For this reason we replace the random number $T_n$ by a random multiparameter in $\nat^d$. We will introduce a natural homeomorphism on the discrete multiparameters that serves as a control mapping for the transfer mechanism. We also aim to consider general state spaces for the random variables $X_{n,k}$. Since we are concerned with weak convergence of probability measures, a requirement on our state space $E$ is that a probability measure $\mu$ on $E$ is uniquely determined by the values $\int_Ef\,d\mu$ for bounded and continuous functions $f:E\to\rr$, i.e.\ Riesz' Theorem applies. Thus we may consider $E$ to be a general metric space; see \cite{Par,Hey2} for details. In Section 2 we present our general transfer theorem and briefly show how it covers classical transfer results appearing in the literature. In Section 3 we focus on three specific examples that give a glance on the full range of possible applications of our general transfer theorem.

\section{Main Result}

Let $\mathcal F=\big\{Y_{\mathbf n}:\,\mathbf n=(n_{1},\ldots,n_{d})\in\nat^d\}$ be a family of $E$-valued random variables, where $E$ is a metric space equipped with its Borel $\sigma$-algebra. We use boldface notation for multiparameters. Let $\phi:\nat^d\to D\subseteq\rr^\ell$ be a homeomorphism.
Note that we consider $\nat^d$ and $D$ as subsets of $\rd$, respectively $\rr^\ell$, with the induced standard topology. Hence any subset of $\nat^d$ is a Borel set and any mapping $\phi:\nat^d\to\rr^\ell$ is continuous. The only assumptions we have to make to ensure that $\phi$ is a homeomorphism are that $\phi$ is injective and that the image $D=\Ima\phi$ is discrete in the sense that intersections of $D$ with sufficiently small open neighborhoods of any $\mathbf x\in D$ coincide with the single point set $\{\mathbf x\}$. Now let $\Delta\subseteq\overline{D}\setminus D$ be a nonempty Borel set of limit points of $D$, where $\overline{D}$ denotes the closure of $D$ in $\rr^\ell$. We assume that for any $\mathbf t=(t_{1},\ldots,t_{\ell})\in\Delta$ and any sequence $(\mathbf N_{n})_{n\in\nat}$ in $\nat^d$ we have
\begin{equation}\label{weakconv}
Y_{\mathbf N_{n}}\Rightarrow\mu_{\mathbf t}\quad
\text{ whenever }\quad\phi(\mathbf N_{n})\to\mathbf t,
\end{equation}
where $\{\mu_{\mathbf t}:\,\mathbf t\in\Delta\}$ is a weakly continuous family of probability measures on $E$. Thus the homeomorphism $\phi$ serves as a control mapping for weak convergence in \eqref{weakconv}. For randomizations of the multiparameter $\mathbf n$ of $\mathcal F$ we are able to give the following general version of a transfer theorem.
\begin{theorem}\label{transfer}
With the above assumptions and notation let $(\mathbf T_{n})_{n\in\nat}$ be a sequence of random variables with values in $\nat^d$ such that $\mathbf T_{n}$ and $\mathcal F$ are independent for every $n\in\nat$. Suppose that $\phi(\mathbf T_{n})\Rightarrow\rho$ for some probability distribution $\rho$ with $\rho(\Delta)=1$. Then we have as $n\to\infty$
\begin{equation}\label{main}
Y_{\mathbf T_{n}}\Rightarrow\int_{\Delta}\mu_{\mathbf t}\,{\rm d}\rho(\mathbf t).
\end{equation}
\end{theorem}
\begin{proof}
For a probability measure $\mu$ and a fixed, bounded and continuous function $f:E\to\rr$ we use the abbreviation
$$\langle\mu,f\rangle=\int_Ef(x)\,{\rm d}\mu(x).$$
Let $\nu(\mathbf n)=P_{Y_{\mathbf n}}$ denote the distribution of $Y_{\mathbf n}$ then $\varphi:D\to\rr$ defined by
$$\varphi(\mathbf t)=\left\langle\nu(\phi^{-1}(\mathbf t)),f\right\rangle\quad\text{ for }\mathbf t\in D$$
is a bounded continuous function, where continuity is due to the fact that $D$ is discrete. We will now extend $\varphi$ to a bounded continuous function $\psi:D\cup\Delta\to\rr$ by
$$\psi(\mathbf t)=\begin{cases}\varphi(\mathbf t) & \text{ if }\mathbf t\in D,\\ \langle\mu_{\mathbf t},f\rangle & \text{ if }\mathbf t\in\Delta.\end{cases}$$
Clearly, $\psi$ is bounded on $D\cup\Delta$, and continuous on $D$ and on $\Delta$, respectively. The latter is due to our assumption on weak continuity of $\mathbf t\mapsto\mu_{\mathbf t}$ on $\Delta$. Now let $\mathbf t\in\Delta$ be arbitrary and let $(\mathbf t_{n})_{n\in\nat}$ be a sequence in $D$ such that $\mathbf t_{n}\to\mathbf t$. Then $\mathbf N_{n}=\phi^{-1}(\mathbf t_{n})$ fulfills the right-hand side of \eqref{weakconv} and thus we get
$$\psi(\mathbf t_{n})=\varphi(\mathbf t_{n})=\left\langle\nu(\phi^{-1}(\mathbf t_n)),f\right\rangle=\left\langle\nu(\mathbf N_n),f\right\rangle\to\left\langle\mu_{\mathbf t},f\right\rangle=\psi(\mathbf t)$$
showing continuity of $\psi$.\\
By the independence assumption together with $\phi(\mathbf T_{n})\Rightarrow\rho$ we finally obtain
\begin{align*}
\left\langle P_{Y_{\mathbf T_n}},f\right\rangle & =\int_{\nat^d}\left\langle\nu(\mathbf n),f\right\rangle\,{\rm d}P_{\mathbf T_{n}}(\mathbf n)=\int_{D}\left\langle\nu(\phi^{-1}(\mathbf t)),f\right\rangle\,{\rm d}P_{\phi(\mathbf T_{n})}(\mathbf t)\\
& =\int_{D\cup\Delta}\psi(\mathbf t)\,{\rm d}P_{\phi(\mathbf T_{n})}(\mathbf t)\to\int_{D\cup\Delta}\psi(\mathbf t)\,{\rm d}\rho(\mathbf t)\\
& =\int_{\Delta}\langle\mu_{\mathbf t},f\rangle\,{\rm d}\rho(\mathbf t)=\left\langle\int_{\Delta}\mu_{\mathbf t}\,{\rm d}\rho(\mathbf t),f\right\rangle.
\end{align*}
This proves our assertion, since the bounded continuous function $f$ is arbitrary.
\end{proof}
We will now apply Theorem \ref{transfer} to the classical convergence scheme of row-sums of triangular arrays of random variables, extending versions of the transfer theorem on second countable locally compact groups by Hazod \cite{Haz,HS} and on separable Banach spaces by Siegel \cite{Sie}. We have to assume second countability to ensure that products of Borel measurable random variables are again Borel measurable; cf.~Proposition 4.1.7 in \cite{Dud}. For a metric space $E$ second countability is equivalent to separability; cf.~Proposition 2.1.4 in \cite{Dud}. Our multiparameter setting allows to extend considerations to random fields as follows. For notational convenience we will write the group multiplication as summation. Note that in case of non-Abelian groups this operation has to be taken in a certain fixed order. In the sequel, relations and operations on multiparameters in $\rd$ are always meant componentwise. Let $E$ be a separable metrizable group. For $n\in\nat$ let $\mathcal F_n=\{X_{n,\mathbf k}:\,\mathbf k\in\nat^d\}$ be a random field, i.e.~a family of $E$-valued random variables, and define for any $n\in\nat$ and $\mathbf N\in\nat^d$
$$Y_{n,\mathbf N}=\sum_{\mathbf k\leq\mathbf N}X_{n,\mathbf k}.$$ 
Suppose the existence of a (componentwise) strictly increasing sequence $(\mathbf k_{n})_{n\in\nat}$ in $\nat^d$ such that for some $\emptyset\not=T\subseteq[0,\infty)^d$ and some family $\{\mu_{\mathbf t}:\,\mathbf t\in T\}$ of probability distributions on $E$ we have as $n\to\infty$
\begin{equation}\label{arrayconvrf}
Y_{n,\lfloor\mathbf k_{n}\mathbf t\rfloor}\Rightarrow\mu_{\mathbf t},
\end{equation}
uniformly on compact subsets of $\mathbf t\in T$, i.e.~for an arbitrary sequence $(\mathbf t_n)_{n\in\nat}$ in $T$ with $\mathbf t_n\to\mathbf t\in T$ we have $\langle P_{Y_{n,\lfloor\mathbf k_n\mathbf t_n\rfloor}},f\rangle\to\langle\mu_{\mathbf t},f\rangle$ for any bounded continuous function $f:E\to\rr$.
\begin{cor}\label{GFrf}
With the above assumptions and notation let $(\mathbf T_{n})_{n\in\nat}$ be a sequence of random variables with values in $\nat^d$ such that $\mathbf T_{n}$ and $\mathcal F_n$ are independent for every $n\in\nat$. Suppose $\mathbf T_{n}/\mathbf k_{n}\Rightarrow\rho$ for some probability distribution $\rho$ with $\rho(T)=1$. Then we have as $n\to\infty$
\begin{equation}\label{transferrf}
Y_{n,\mathbf T_{n}}\Rightarrow\int_{T}\mu_{\mathbf t}\,{\rm d}\rho(\mathbf t).
\end{equation}
\end{cor}
\begin{proof}
Let $\mathcal F=\{Y_{n,\mathbf N}:\,n\in\nat,\,\mathbf N\in\nat^d\}$ and note that without loss of generality we may assume that $\mathbf T_{n}$ and $\mathcal F$ are independent for every $n\in\nat$. Otherwise we may change to random variables $\mathbf T_{n}'$ independent of $\mathcal F$ and with the same distribution as $\mathbf T_{n}$ which leave the assertion in \eqref{transferrf} unchanged. One may easily check that $\phi:\nat^{1+d}\to D$ defined by $\phi(n,\mathbf N)=(\frac1n,\frac{\mathbf N}{\mathbf k_{n}})$ with $D=\Ima\phi$ is a homeomorphism. For $\Delta=\{0\}\times T\subseteq\overline{D}\setminus D$ and any sequence $(n_{k},\mathbf N_{k})\in\nat^{1+d}$ with $\phi(n_{k},\mathbf N_{k})\to(0,\mathbf t)\in\Delta$ we have $n_{k}\to\infty$ and $\mathbf m_k=\mathbf N_{k}/\mathbf k_{n_{k}}\to\mathbf t$ so that by \eqref{arrayconvrf} we have
$$Y_{n_{k},\mathbf N_{k}}=Y_{n_{k},\lfloor\mathbf k_{n_k}\mathbf m_k\rfloor}\Rightarrow\mu_{\mathbf t}.$$
Note that the family $\{\mu_{0,\mathbf t}=\mu_{\mathbf t}:\,(0,\mathbf t)\in\Delta\}$ is weakly continuous due to the uniform compact convergence in \eqref{arrayconvrf}.
Hence \eqref{weakconv} is fulfilled and we further have $\phi(n,\mathbf T_{n})=\big(\frac1n,\frac{\mathbf T_{n}}{\mathbf k_{n}}\big)\Rightarrow\delta_{0}\otimes\rho$. An application of Theorem \ref{transfer} now gives
$$Y_{n,\mathbf T_{n}}\Rightarrow\int_{\Delta}\mu_{0,\mathbf t}\,{\rm d}(\delta_{0}\otimes\rho)(s,\mathbf t)=\int_{T}\mu_{\mathbf t}\,{\rm d}\rho(\mathbf t)$$
concluding the proof.
\end{proof}
In case $d=1$ and $E=\rr^m$ the classical transfer theorem of Gnedenko \cite{Gne,GF} is a special case of Corollary \ref{GFrf} as follows. Suppose $(X_{n,k})_{k\in\nat}$ is a sequence of i.i.d.~random variables for every $n\in\nat$ such that for some strictly increasing sequence $(k_{n})_{n\in\nat}$ in $\nat$ we have
$$Y_{n,k_n}=\sum_{k=1}^{k_n}X_{n,k}\Rightarrow\mu,$$
where $\mu$ is a necessarily infinitely divisible probability distribution on $\rr^m$. Then \eqref{arrayconvrf} is fulfilled, where $(\mu_t)_{t\geq0}$ is the continuous convolution semigroup generated by $\mu$. Thus in this situation Corollary \ref{GFrf} restates the transfer theorem of Gnedenko as presented in the Introduction. Note that the proof in \cite{Gne,GF} is given in terms of convergence of Fourier transforms. Further note that instead of the above partial sums it is also possible to consider partial (componentwise) maxima of random vectors \cite{Ber,Aks1,Aks2} or partial products of random variables \cite{KR}.
\begin{remark}
Our random fields approach in Corollary \ref{GFrf} can be further applied to the convergence of finite-dimensional marginal distributions of certain randomly stopped partial sum processes on $E$ (e.g., see Theorem 4.1 in \cite{BMS} for $E=\rr^m$). It is also possible to apply Theorem \ref{transfer} to obtain a transfer type result for randomly stopped stochastic processes on the path space $E=D([0,\infty),S)$ of c\`adl\`ag functions on a complete separable metric space $S$, since $D([0,\infty),S)$ can itself be seen as a complete separable metric space; see \cite{Bil,Kal} for details. 
\end{remark}

\section{Examples}

To mirror the full range of possible applications of our Theorem \ref{transfer} we aim to give further examples in which different limits under different regimes arise with a mixture distribution of the transfer type limit. 

\subsection{A transfer theorem for negatively associated random fields.}

Our first example is a generalization of Gnedenko's transfer theorem for partial sums in a multiparameter and non-i.i.d.\ setting. Let $\{X_{\mathbf n}:\,\mathbf n\in\nat^d\}$ be a field of stationary, centered and negatively associated random variables with positive and finite second moment. Stationarity means that $\{X_{\mathbf n+\mathbf m}:\,\mathbf n\in\nat^d\}$ is distributed as $\{X_{\mathbf n}:\,\mathbf n\in\nat^d\}$ for every $\mathbf m\in\nat^d$. Negative association was introduced in \cite{JDP} and claims that
$$\Cov\big(f(X_{\mathbf n}:\,\mathbf n\in I),g(X_{\mathbf m}:\,\mathbf m\in J) \big)\leq0$$
for every pair $I,J$ of disjoint subsets of $\nat^d$ and any coordinatewise increasing functions $f,g$ with $\Exp[f^2(X_{\mathbf n}:\,\mathbf n\in I)]<\infty$ and $\Exp[g^2(X_{\mathbf m}:\,\mathbf m\in J)]<\infty$. Let us define
$$Y_{\mathbf n,\mathbf N}=\frac1{\sqrt{|\mathbf n|}}\sum_{\mathbf k\leq\mathbf N}X_{\mathbf k}$$
for $\mathbf n,\mathbf N\in\nat^d$, where $|\mathbf n|=\prod_{i=1}^dn_i$. Introduce a homeomorphism $\phi:\nat^{2d}\to D$ by
$$\phi(\mathbf n,\mathbf N)=\left(\frac1{\mathbf n},\frac{\mathbf N}{\mathbf n}\right)$$
with $D=\Ima(\phi)\subseteq(0,\infty)^{2d}$. Let $\Delta=\{\mathbf 0\}\times[0,\infty)^d\subseteq\overline{D}\setminus D$ then for any sequence $(\mathbf n_k,\mathbf N_k)\in\nat^{2d}$ with $\phi(\mathbf n_k,\mathbf N_k)\to(\mathbf 0,\mathbf t)\in\Delta$ we necessarily have $\mathbf n_k\to\mathbf\infty$ (componentwise) and $\mathbf m_k=\mathbf N_k/\mathbf n_k\to\mathbf t$. It follows from Theorem 1 in \cite{ZW} that
$$Y_{\mathbf n_k,\mathbf N_k}=\frac1{\sqrt{|\mathbf n_k|}}\sum_{\mathbf m\leq\lfloor\mathbf n_k\mathbf m_k\rfloor}X_{\mathbf m}\Rightarrow\mu_{\mathbf t},$$
where $\mu_{\mathbf t}=\mathcal N_{0,\sigma^2|\mathbf t|}$ has a Gaussian distribution for some $\sigma^2>0$ depending on the covariance structure of the random field $\{X_{\mathbf n}:\,\mathbf n\in\nat^d\}$. Clearly, the family of probability distributions $\{\mu_{\mathbf 0,\mathbf t}=\mu_{\mathbf t}:\,(\mathbf 0,\mathbf t)\in\Delta\}$ is weakly continuous. An application of our Theorem 1 directly leads to the following result, the same way we have proven Corollary 2.

\begin{cor}
Let $(\mathbf T_k)_{k\in\nat}$ be a sequence of $\nat^d$-valued random variables such that $\mathbf T_k/\mathbf n_k\Rightarrow\rho$ for some sequence $\mathbf n_k\to\mathbf\infty$ in $\nat^d$ and some probability distribution $\rho$ on $[0,\infty)^d$. Then we have
$$\frac1{\sqrt{|\mathbf n_k|}}\sum_{\mathbf m\leq\lfloor\mathbf n_k\mathbf T_k\rfloor}X_{\mathbf m}\Rightarrow\int_{[0,\infty)^d}\mathcal N_{0,\sigma^2|\mathbf t|}\,d\rho(\mathbf t).$$
\end{cor}

\subsection{A transfer theorem for semistable domains of attraction.}

In case $d=1$ and $E=\rr^m$ we consider the special case of $X_{n,k}=A_{n}(X_{k}-a_{n})$ for a sequence $(X_{k})_{k\in\nat}$ of i.i.d.\ random vectors on $\rr^m$, invertible linear operators $A_{n}\in\GL(\rr^m)$ and shifts $a_{n}\in\rr^m$. Suppose there exists a strictly increasing sequence $k_{n}\uparrow\infty$ in $\nat$ such that $k_{n+1}/k_{n}\to c\geq1$ 
\begin{equation}\label{arrayconvrn}
A_{n}\sum_{k=1}^{k_{n}}(X_{k}-a_{n})\Rightarrow\mu
\end{equation}
for some full probability distribution $\mu$ on $\rr^m$, where full means not supported on any proper hyperplane of $\rr^m$. Then $\mu$ is infinitely divisible (in particular it is operator semistable) and it is known that the normalizing operators $A_{n}$ can be chosen such that there exist $E\in\GL(\rr^m)$ and $B_{n}\in\GL(\rr^m)$ with $B_{k_{n}}=A_{n}$ fulfilling $B_{\lfloor\lambda n\rfloor}B_{n}^{-1}\to\lambda^{-E}=\sum_{k=0}^\infty\frac{(-\log\lambda)^k}{k!}\,E^k$ uniformly on compact subsets of $\{\lambda>0\}$. Hence the sequence $(B_{n})_{n\in\nat}$ is regularly varying with exponent $-E$, see \cite{MS} for the details. The case $c=1$ corresponds to operator stability $\mu^t=t^E\mu\ast\delta_{a(t)}$ for all $t>0$ and some continuous function $t\mapsto a(t)$. In case $c>1$  operator semistability is given by
\begin{equation}\label{oss}
\mu^c=c^E\mu\ast\delta_a\quad\text{ for some }a\in\rr^m;
\end{equation}
see \cite{MS} for the details. Moreover, the shifts $(a_n)_{n\in\nat}$ can be chosen to be embeddable into a sequence $(b_n)_{n\in\nat}\subseteq\rr^m$ with $b_{k_n}=a_n$ such that $B_n\sum_{k=1}^n(X_k-b_n)$ is weakly relatively compact with weak limit points
$$\big\{\mu_t=t^{-E}(\mu^t\ast\delta_{-a(t)}):\,t\in[1,c)\big\},$$
where $t\mapsto a(t)$ is continuous with $a(1)=0$ and $a(t)\to a$ as $t\uparrow c$. Especially, $\mu_1=\mu$ and $\mu_t\to\mu=:\mu_c$ weakly as $t\uparrow c$ by \eqref{oss}.  A limit $\mu_t$ arises for subsequences of the form $n=k_{p_n}t_n$ with $t_n\to t\in[1,c]$, where $p_n\in\nat$ is given by $k_{p_n}\leq n<k_{p_n+1}$; see Section 4 in \cite{pbk3} for details. An univariate version ($m=1$) of the above results is given in \cite{CM}, where centering shifts are constructed via the quantile function. 

Although $B_n\sum_{k=1}^n(X_k-b_n)$ does not converge in distribution, it is possible to achieve a limit distribution for random sums by a transfer type theorem as shown in \cite{pbk6b}. For $n\in\nat$ and the increasing sampling sequence $(k_n)_{n\in\nat}$ define $k_0=1$ and
\begin{equation}\label{mantissa}
\psi(n)=t_n\quad\text{ if }n=k_{p_n}t_n\text{ with }p_n\in\nat_0\text{ given by }k_{p_n}\leq n<k_{p_n+1}.
\end{equation}
\begin{cor}\label{GFmant}
Let $(T_{n})_{n\in\nat}$ be a sequence of positive integer valued random variables  independent of $(X_{k})_{k\in\nat}$ and assume that $\psi(T_n)\Rightarrow\rho$ for some probability distribution $\rho$ on $[1,c]$. Then we have
$$B_{T_n}\sum_{k=1}^{T_{n}}\big(X_{k}-b_{T_n}\big)\Rightarrow\int_1^c\mu_t\,{\rm d}\rho(t)=\int_1^ct^{-E}(\mu^t\ast\delta_{-a(t)})\,{\rm d}\rho(t).$$
\end{cor}
\begin{proof}
Let $\mathcal F=\big\{Y_{n}=B_n\big(\sum_{k=1}^{n}X_{k}-b_{n}\big):\,n\in\nat\big\}$ and let $\phi:\nat\to D$ be given by  $\phi(n)=(\frac1n,\psi(n))$ using \eqref{mantissa}. Then for $\Delta=\{0\}\times[1,c]=\overline{D}\setminus D$ and any sequence $n_{k}\in\nat$ with $\phi(n_{k})\to(0,t)\in\Delta$ we have $n_{k}\to\infty$ and $\psi(n_k)\to t$ so that we get
$$Y_{n_{k}}=B_{n_{k}}\sum_{\ell=1}^{n_{k}}\big(X_{\ell}-b_{n_{k}}\big)\Rightarrow\mu_t=t^{-E}(\mu^t\ast\delta_{-a(t)})$$
as described above. Hence \eqref{weakconv} is fulfilled for the weakly continuous family of distributions $\big\{\mu_{0,t}=\mu_t=t^{-E}(\mu^t\ast\delta_{-a(t)}):\,(0,t)\in\Delta\big\}$, where $\mu_1=\mu=\mu_c$. Further the assumptions of Corollary \ref{GFmant} imply $\phi(T_{n})\Rightarrow\delta_{0}\otimes\rho$ and hence by Theorem \ref{transfer} the assertion follows.
\end{proof}
As an example, it is implicitely shown in \cite{pbk6b} that for any $c>1$ the logarithmic summation with
$$P\{T_n=k\}=\begin{cases}D_n^{-1}\frac1k & \text{ if }1\leq k\leq n\\ 0 & \text{ else}\end{cases}\quad\text{ with }\quad D_n=\sum_{k=1}^n\frac1k\sim\log n$$
fulfills $\psi(T_n)\Rightarrow\rho$ for the logarithmic distribution $\rho$ on $[1,c]$ with probability density $t\mapsto(t\log c)^{-1}1_{[1,c]}(t)$. Hence by Corollary \ref{GFmant} we get for the distributions $\nu_k$ of $B_k\sum_{\ell=1}^k(X_\ell-b_k)$
\begin{equation}\label{logsum}
\frac1{\log n}\sum_{k=1}^n\frac1k\,\nu_k\to\frac1{\log c}\int_1^ct^{-E}(\mu^t\ast\delta_{-a(t)})\,\frac{{\rm d}t}{t}\quad\text{ weakly as } n\to\infty.
\end{equation}
Further examples of random variables $T_n$ (corresponding to summability methods) with $\psi(T_n)$ converging to the logarithmic distribution are given in \cite{pbk6b}.

\subsection{Transfer theorems for random allocations}

Let $n$ balls be allocated to $N$ boxes independently and with equal probability. Denote by $\mu_{r}(n,N)$ the final number of boxes containing exactly $r\geq0$ balls and let
$$\mu_{r}^\ast(n,N)=\frac{\mu_{r}(n,N)-\Exp[\mu_{r}(n,N)]}{\sqrt{\SD^2[\mu_{r}(n,N)]}},\quad0\leq r\leq n$$
denote the standardized random variables. The limit behaviour of $\mu_{r}^\ast(n,N)$ as simultaneously $n,N\to\infty$ is well known. The possible limit distributions are either the standard normal distribution $\mathcal N_{0,1}$ or a standardized Poisson distribution $\pi_{\lambda}^\ast$ with parameter $\lambda>0$. For details we refer to the monograph \cite{KSC} and the literature cited therein. For random numbers of balls and boxes we can derive the following transfer theorems as applications of Theorem \ref{transfer}. Let $\{(T_{n},U_{N}):\,n,N\in\nat\}$ be random variables in $\nat^2$ such that $(T_{n},U_{N})$ is independent of the allocations and thus of $\mathcal F_{r}=\{\mu_{r}^\ast(n,N):\,n,N\in\nat, n\geq r\}$ for every $(n,N)\in\nat^2$ and every $r\geq0$. The following transfer theorem for the number of empty boxes $\mu_{0}(n,N)$ is already known by \cite{pbk}. Its method of proof inspired our general version, but the proof of Theorem \ref{transfer} is even simpler than the proof given for the special result in \cite{pbk}. In turn Theorem \ref{transfer} can be applied to derive transfer theorems for $\mu_{r}(n,N)$ with $r\geq1$. Partial results can already be found in \cite{IMS}.
\begin{cor}\label{ttra0}
Let $\Delta_{0}=\big(\rr_{+}\times\{0\}\big)\cup\big(\{0\}\times\rr_{+}\big)$ and for any $(g,d)\in\Delta_{0}$ let
$$\nu_{g,d}=\begin{cases}\pi_{1/g}^\ast & \text{ if }0<g<\infty,\,d=0,\\ \mathcal N_{0,1} & \text{ if }g=d=0,\\ \pi_{1/d}^\ast & \text{ if }g=0,\,0<d<\infty,\end{cases}$$
which defines a weakly continuous family $\{\nu_{g,d}:\,(g,d)\in\Delta_{0}\}$ of probability distributions. Assume that for some probability distribution $\rho_{0}$ on $\Delta_{0}$
$$\phi_{0}(T_{n},U_{N})=\Big(\frac{2U_{N}}{T_{n}^2},\frac{e^{T_{n}/U_{N}}}{U_{N}}\Big)\Rightarrow\rho_{0}.$$ 
Then we have
$$\mu_{0}^\ast(T_{n},U_{N})\Rightarrow\int_{\Delta_{0}}\nu_{g,d}\,{\rm d}\rho_{0}(g,d).$$
\end{cor}
\begin{cor}\label{ttra1}
Let $\Delta_{1}=\{0\}\times\rr_{+}$ and for any $(0,d)\in\Delta_{1}$ let
$$\nu_{d}=\begin{cases}\mathcal N_{0,1} & \text{ if }d=0,\\ \pi_{1/d}^\ast & \text{ if }0<d<\infty,\end{cases}$$
which defines a weakly continuous family $\{\nu_{d}:\,(0,d)\in\Delta_{1}\}$ of probability distributions. Assume that for some probability distribution $\rho_{1}$ on $\rr_+$
$$\phi_{1}(T_{n},U_{N})=\Big(\frac{U_{N}}{T_{n}^2},\frac{e^{T_{n}/U_{N}}}{T_{n}}\Big)\Rightarrow\delta_0\otimes\rho_{1}.$$ 
Then we have
$$\mu_{1}^\ast(T_{n},U_{N})\Rightarrow\int_{0}^\infty\nu_{d}\,{\rm d}\rho_{1}(d).$$
\end{cor}
\begin{cor}\label{ttrar}
For $r\geq2$ let $\Delta_{r}=\{0\}\times\rr_{+}$. For any $(0,d)\in\Delta_{r}$ again let
$$\nu_{d}=\begin{cases}\mathcal N_{0,1} & \text{ if }d=0,\\ \pi_{1/d}^\ast & \text{ if }0<d<\infty.\end{cases}$$
Assume that for some probability distribution $\rho_{r}$ on $\rr_+$
$$\phi_{r}(T_{n},U_{N})=\Big(\frac{1}{T_{n}},\frac{r!U_{N}^{r-1}}{T_{n}^r}\,e^{T_{n}/U_{N}}\Big)\Rightarrow\delta_0\otimes\rho_{r}.$$ 
Then we have
$$\mu_{r}^\ast(T_{n},U_{N})\Rightarrow\int_{0}^\infty\nu_{d}\,{\rm d}\rho_{r}(d).$$
\end{cor}
\begin{proof}[Proof of Corollaries \ref{ttra0}--\ref{ttrar}]
For any sequences $(n_{k})_{k\in\nat}$ and $(N_{k})_{k\in\nat}$ of positive integers such that $\phi_{r}(n_{k},N_{k})\to(g,d)\in\Delta_{r}$ for some $r\geq0$ it can easily be shown that $n_{k}\to\infty$, $N_{k}\to\infty$ and that $\mu_{r}^\ast(n_{k},N_{k})$ converges in distribution to the corresponding normal or Poissonian distribution given in Corollaries \ref{ttra0}--\ref{ttrar}. For $r=0$ this is already shown in the proof of Theorem 2.1 in \cite{pbk} and for $r\geq1$ it follows in the same spirit using the limit theorems given in \cite{KSC}. Hence for an application of Theorem \ref{transfer} we simply have to argue that $\phi_{r}:\nat^2\to D=\Ima\phi_r$ is a homeomorphism for any $r\geq0$. Again, for $r=0$ these arguments have been given in the proof of Theorem 2.1 in \cite{pbk} and for $r\geq1$ they can also be derived easily.
\end{proof}

\bibliographystyle{plain}

\begin{thebibliography}{10}

\bibitem{Aks1} Aksomaitis, A. (1992) On convergence of the maximum of dependent random variables. {\it Lith.~Math.~J.} {\bf 32} 1--3.

\bibitem{Aks2} Aksomaitis, A. (1994) Rate of convergence in the limit transfer theorem of the maximum and minimum. In: Grigelionis, B.~(ed.) {\it Probability Theory and Mathematical Statistics.} Proceedings of the 6th Vilnius conference, 1993. VSP, Utrecht, pp.~35--42.


\bibitem{pbk3} Becker-Kern, P. (2003) Stable and semistable hemigroups: Domains of attraction and selfdecomposability. {\it J.~Theoret.~Probab.} {\bf 16} 573--598.

\bibitem{pbk} Becker-Kern, P. (2007) An almost sure limit theorem for mixtures of domains in random allocation.  {\it Studia Sci.~Math.~Hungar.} {\bf 44} 331--354.

\bibitem{pbk6b} Becker-Kern, P. (2007) Almost sure limit theorems of mantissa type for semistable domains of attraction. {\it Acta Math.~Hungar.} {\bf 114} 301--336.

\bibitem{BMS} Becker-Kern, P.; Meerschaert, M.M.; and Scheffler, H.P. (2004) Limit theorem for continuous-time random walks with two time scales. {\it J.~Appl.~Probab.} {\bf 41} 455-466.


\bibitem{BK} Bening, V.E.; and Korolev, V.Y. (2002) {\it Generalized Poisson Models and their Applications in Insurance and Finance.} VSP, Utrecht.


\bibitem{Ber} Berman, S.M. (1962) Limiting distribution of the maximum term in sequences of dependent random variables. {\it Ann.~Math.~Statist.} {\bf 33} 894--908.

\bibitem{Bil} Billingsley, P. (1999) {\it Convergence of Probability Measures.} 2nd Edition, Wiley, New York.

\bibitem{CM} Cs\"org\H o, S.; and Megyesi, Z. (2003) Merging to semistable laws. {\it Theory Probab.\ Appl.} {\bf 47} 17--33.

\bibitem{DY} Dion, J.P.; and Yanev, N.M. (1997) Limit theorems and estimation theory for branching processes with an increasing random number of ancestors. {\it J.~Appl.~Probab.} {\bf 34} 309--327.

\bibitem{Dud} Dudley, R.M. (2002) {\it Real Analysis and Probability.} Cambridge University Press, Cambridge.

\bibitem{FKT} Finkelstein, M.; Kruglov, V.M.; and Tucker, H.G. (1994) Convergence in law of random sums with non-random centering. {\it J. Theoret. Probab.} {\bf 7} 565--598.

\bibitem{Gne} Gnedenko, B.V. (1983) On limit theorems for a random number of random variables. In: Proceedings of the $4th$ USSR-Japan Symposium on Probability Theory and Mathematical Statistics, Tiblisi, 1982. {\it Lecture Notes Math.} {\bf 1021} Springer, Berlin, pp.~167--176.

\bibitem{GF} Gnedenko, B.V.; and Fahim, H. (1969) On a transfer theorem. {\it Soviet Math.~Dokl.} {\bf 10} 769--772; translation from {\it Dokl.~Akad.~Nauk SSSR} {\bf 187} 15--17.

\bibitem{GK} Gnedenko, B.V.; and Korolev, V.Y. (1996) {\it Random Summation.} CRC Press, Boca Raton.

\bibitem{Gut} Gut, A. (1988) {\it Stopped Random Walks.} Springer, New York.

\bibitem{Haz} Hazod, W. (1994) On the limit behavior of products of a random number of group valued random variables. {\it Theory Probab.~Appl.} {\bf 39} 249--263.


\bibitem{HS} Hazod, W.; and Siebert, E. (2001) {\it Stable Probability Measures on Euclidean Spaces and on Locally Compact Groups. Structural Properties and Limit Theorems.} Kluwer, Dordrecht.


\bibitem{Hey2} Heyer, H. (2010) {\it Structural Aspects in the Theory of Probability.} 2nd Edition, World Scientific, Singapore.

\bibitem{IMS} Ivchenko, G.I.; Medvedev, Y.I.; and Sevast'yanov, B.A. (1967) The distribution of a random number of particles in cells. {\it Math. Notes} {\bf 1}, 363--366; translation from {\it Mat. Zametki} {\bf 1} 549--554.

\bibitem{JDP} Joag-Dev, K.; and Proschan, F. (1983) Negative association of random variables, with applications. {\it Ann. Statist.} {\bf 11} 286--295.

\bibitem{Kal} Kallenberg, O. (2002) {\it Foundations of Modern Probability.} 2nd Edition, Springer, New York.

\bibitem{KSC} Kolchin, V.F.; Sevast'yanov, B.A.; and Chistyakov, V.P. (1978) {\it Random Allocations.} Winston, Washington.

\bibitem{KK1} Korolev, V.Y.; and Kossova, E.V. (1994) Limit distributions of randomly indexed multidimensional random sequences with operator normalization. {\it J.~Math.~Sci.~(New York)} {\bf 72} 2915--2929.

\bibitem{KK2} Korolev, V.Y.; and Kossova, E.V. (1995) Convergence of multidimensional random sequences with independent random indices. {\it J.~Math.~Sci.~(New York)} {\bf 76} 2259--2268.

\bibitem{KKru} Korolev, V.Y.; and Kruglov, V.M. (1993) Limit theorems for random sums of independent random variables. In: Kalashnikov, V.V.~et  al.~(eds.) Stability Problems for Stochastic Models. Proceedings of the 14th Seminar, Suzdal, 1991. {\it Lecture Notes Math.} {\bf 1546} Springer, Berlin, pp.~100--120.

\bibitem{KR} Kowalski, P.; and Rychlik, Z. (1995) On the products of a random number of independent random variables. {\it Bull.~Pol.~Acad.~Sci. Math.} {\bf 43} 219--230.

\bibitem{MS} Meerschaert, M.M.; and Scheffler, H.P. (2001) {\it Limit Distributions for Sums of Independent Random Vectors.} Wiley, New York.




\bibitem{Par} Parthasarathy, K.R. (1967) {\it Probability Measures on Metric Spaces.} Academic Press, New York.

\bibitem{Rob} Robbins, H. (1948) The asymptotic distribution of the sum of a random number of random variables. {\it Bull.~Amer.~Math.~Soc.} {\bf 54} 1151--1161.

\bibitem{Sie} Siegel, G. (1988) Convergence of randomly selected sums in a separable Banach space. {\it Math. Nachr.} {\bf 139} 139--153.

\bibitem{Sil} Silvestrov, D.S. (2004) {\it Limit Theorems for Randomly Stopped Stochastic Processes.} Springer, London.



\bibitem{Tho} Thomas, D.I. (1972) On limiting distributions of a random number of dependent random variables. {\it Ann.~Math.~Statist.} {\bf 43} 1719--1726.



\bibitem{ZW} Zhang, L.X.; and Wen, J (2000) A weak convergence for negatively associated fields. {\it Statist. Probab. Lett.} {\bf 53} 259--267.

\end{thebibliography}

\end{document}